%% file: main.tex
\documentclass[letterpaper,11pt,reqno]{amsart}

\usepackage[utf8]{inputenc}
\usepackage[T1]{fontenc}
\usepackage{mathcommand}
\usepackage{xcolor}
\usepackage{xstring}
\usepackage[pagebackref=true]{hyperref}
\usepackage{fancyhdr}
\usepackage{geometry}
\usepackage{amsmath}
\usepackage{amsthm}
\usepackage{amssymb}
\usepackage{mathtools}
\usepackage{stmaryrd}
\usepackage{mathdots}
\usepackage{faktor}
\usepackage{framed}
\usepackage{hyphenat}
\usepackage[capitalize]{cleveref}
\usepackage{array}
\usepackage{enumitem}
\usepackage[all,cmtip]{xy}
\usepackage{tikz}
\usetikzlibrary{cd}
\usepackage{microtype}

\definecolor{verydarkblue}{rgb}{0,0,0.5}

\hypersetup{
    breaklinks,
    colorlinks,
    citecolor=verydarkblue,
    linkcolor=verydarkblue,
    urlcolor=verydarkblue,
}

\theoremstyle{plain}

\newtheorem{introthm}{Theorem}

\newtheorem{thm}{Theorem}[section]
\newtheorem{prop}[thm]{Proposition}
\newtheorem{lemma}[thm]{Lemma}
\newtheorem{cor}[thm]{Corollary}

\theoremstyle{definition}

\newtheorem{definition}[thm]{Definition}

\newtheorem{remark}[thm]{Remark}
\newtheorem{setup}[thm]{Setup}

\IfFileExists{./article-style.tex}{\input{article-style.tex}}{}

\newcommand{\NashO}{\mathrm{Nash}_1}
\newcommand{\NashT}{\mathrm{Nash}_2}
\newcommand{\Nashn}{\mathrm{Nash}_n}
\newcommand{\Hilb}{\mathrm{Hilb}_{\binom{n+d}{d}}}
\newcommand{\Hilbl}{\mathrm{Hilb}_l}
\newcommand{\tors}{\mathrm{tors}}
\newcommand{\Nash}{\mathrm{Nash}}
\newcommand{\Grass}{\mathrm{Grass}}

\newcommand{\mysetminusD}{\hbox{\tikz{\draw[line width=0.6pt,line cap=round] (3pt,0) -- (0,6pt);}}}
\newcommand{\mysetminusT}{\mysetminusD}
\newcommand{\mysetminusS}{\hbox{\tikz{\draw[line width=0.45pt,line cap=round] (2pt,0) -- (0,4pt);}}}
\newcommand{\mysetminusSS}{\hbox{\tikz{\draw[line width=0.4pt,line cap=round] (1.5pt,0) -- (0,3pt);}}}

\newcommand{\mysetminus}{\mathbin{\mathchoice{\mysetminusD}{\mysetminusT}{\mysetminusS}{\mysetminusSS}}}

\begin{document}

\title{Higher Nash Blow-ups and Nobile's Theorem}

\author{Shravan Saoji}

\address[S.\ Saoji]{%
    Department of Mathematics\\
    University of Oklahoma\\
    601 Elm Avenue, Room 423\\
    Norman, OK 73019 (USA)%
}

\email{shravan.saoji@ou.edu}


\begin{abstract}
We study the higher Nash blow-ups introduced by T.\ Yasuda and investigate the
higher version of the classical Nobile's theorem. In particular, we give a
characteristic free proof of the higher Nobile's theorem for the graded case. We
also give a proof for the 2nd order Nash blow-ups in characteristic zero.
\end{abstract}

\maketitle


\section{Introduction}

The higher Nash blow-up was first introduced by T.~Yasuda \cite{MR2371378} which can be thought of as a higher version of the Nash blow-up of a variety. The usual Nash blow-up of a variety replaces each smooth point by the limit of nearby tangent spaces in a Grassmannian. The higher Nash blow-up of a variety instead replaces each smooth point by the limit of nearby infinitesimal neighborhoods in a Hilbert scheme of points. The Nash blow-up and later higher Nash blow-up were introduced as a method for resolution of singularities. An important step towards this problem is to prove that Nash and higher Nash blow-ups modify singular varieties. In the case of the Nash blow-up, there is a classical theorem by A. Nobile \cite{MR0409462} for varieties over $\mathbb{C}$. The theorem says that the Nash blow-up of a variety is isomorphic to the variety if and only if the variety is non-singular. The same theorem was proved by D.~Duarte and L.~Núñez-Betancourt \cite{MR4382471} for normal varieties over a field of positive characteristic. We investigate the higher version of Nobiles's theorem for the graded case. The higher Nobile's theorem has been investigated for curves in \cite{MR2371378}, and for normal hypersurfaces in \cite{MR3614150}, and \cite{MR4382471}. 
We prove the following theorems.

\begin{introthm}[see~\cref{Higher_Nash_graded}] Let $X$ be a graded algebraic variety over an algebraically closed field $k$. Suppose $\Nashn(X) \cong X$ for some $n \in \mathbb{N}$, then 
    \begin{enumerate}[label=(\arabic*)]
        \item if $\operatorname{char}(k) = 0$, then $X$ is a non-singular variety.
        \item if $X$ is normal and $\operatorname{char}(k) > 0$, then $X$ is a non-singular variety.
    \end{enumerate}
\end{introthm}

We also give the higher version of Nobile's theorem for the 2nd order Nash
blow-ups.


\begin{introthm}[see~\cref{Higher_Nash_n=2}] Let $X$ be an algebraic variety over an algebraically closed field $k$ with $\operatorname{char}(k) = 0$. Suppose $\NashT(X) \cong X$, then $X$ is a non-singular variety.
\end{introthm}

 We show that theorems A and B are achieved as corollaries to the following theorem.

\begin{introthm}[see~\cref{NashntoNashO}]
Assume the hypothesis of theorem A or theorem B, then $\NashO(X) \cong X$.
\end{introthm}

We can say that the Nobiles's theorem implies the higher Nobile's theorem in
these cases.


\subsection*{Terminology and assumptions}
Throughout the paper, we will assume $k$ to be an algebraically closed field of
arbitrary characteristic unless specified. All our varieties will be reduced and irreducible
separated schemes of finite type over $k$. We say an affine variety $X \subseteq \mathbb A^s$ is graded if the defining ideal of $X$ is generated by homogeneous elements.
 
\section{Preliminaries}

In this section, we give the definition, some properties, and facts of the module of principal parts of order $n$. Further, we introduce the notion of higher Nash blow-up.
\vspace{1ex}

\subsection*{The module of principal parts}
~\vspace{1ex}

\noindent The module of principal parts is a generalized version of the module of K\"{a}hler differentials.

\begin{definition} \cite{MR0238860}
    Let $X$ be a Noetherian $k$ scheme. Consider the diagonal morphism $\Delta: X \longrightarrow X \times_k X$ and $\mathcal{I}_{\Delta}$ be the ideal sheaf defining the diagonal $\Delta(X) \subset X \times_{k} X$. We define the sheaf of principal parts of order $n$, denoted $\mathcal{P}_{X/k}^n$, as
    \begin{align*}
        \mathcal{P}_{X/k}^n := \mathcal{O}_{X \times_{k} X}/\mathcal{I}_{\Delta}^{n+1}.
    \end{align*}
     Similarly, define $\mathcal{P}_{X/k,+}^n := \mathcal{I}_{\Delta}/\mathcal{I}_{\Delta}^{n+1}$. It is easy to see that $\mathcal{P}_{X/k,+}^1 \cong \Omega_{X/k}$.
Now, let's define the module of principal parts of order $n$ for a $k$-algebra. When $X = \operatorname{Spec}R$ for a $k$-algebra $R$, the construction is as follows. Consider the multiplication map 
\[
\varphi: R \otimes R \longrightarrow R, \quad a \otimes b \longmapsto ab.
\]
    Then $\mathcal{P}_{X/k}^n$ corresponds to
    \[
    P_{R/k}^n := {R \otimes R}/{I_{R}^{n+1}}.
    \]
    where $I_R:= \text{ker}(\varphi)$. Similarly, define the $R$-module $P_{R/k,+}^{n} := {I_R}/{I_{R}^{n+1}}$ and, we have the following isomorphism
\[
P_{R/k}^n \cong R \oplus P_{R/k,+}^{n}.
\]
Notice that when $n=1$, $P_{R/k,+}^1 \cong \Omega_{R/k}$. Consider the canonical derivation of order $n$
\[
d_{R}^n : R \longrightarrow P_{R/k,+}^{n}, \quad x \longmapsto (1 \otimes x - x \otimes 1) + I_{R}^{n+1}.
\]
The $R$-module $P_{R/k,+}^{n}$ is generated by the image of the map $d_{R}^n$.
\end{definition}

To lighten the notation, we will write $\mathcal{P}_{X}^n = \mathcal{P}_{X/k}^n$, and $P_{R,+}^n = P_{R,+}^{n}$ when the base field is understood from the context. Similarly, we will write $d^n = d_{R}^n$ when $R$ is clear from the context. We now discuss a fact about canonical derivation, and the functoriality of the module of principal parts.

\begin{remark} \cite{MR3987965} \label{HigherLeibniz}
    Let $x, y \in R$, then a simple computation shows that
    \begin{align*}
        {d^n}(xy) = {d^n}(x){d^n}(y) + y{d^n}(x) + x{d^n}(y).
    \end{align*}
\end{remark}

\begin{remark}
    Let $R$ and $S$ be $k$-algebras and $f: R \longrightarrow S$ be a $k$-algebra homomorphism. We see $S$ as a $R$-module, then we have the following $R$-module homomorphisms
    \begin{align*}
       P_{R}^n \longrightarrow P_{S}^n
    \end{align*}
    and
    \begin{align*}
        P_{R,+}^n \longrightarrow P_{S,+}^n.
    \end{align*}
Similarly, we have the following $S$-module homomorphisms
\begin{align*}
        P_{R}^n \otimes_R S \longrightarrow P_{S}^n
    \end{align*}
    and
    \begin{align*}
        P_{R,+}^n \otimes_R S \longrightarrow P_{S,+}^n.
    \end{align*}
\end{remark}
\vspace{1ex}

\subsection*{Higher Nash blow-ups}
~\vspace{1ex}

\noindent We will now define the notion of higher Nash blow-up which is a higher version of the Nash blow-up of a variety.

\begin{definition} (Hilbert scheme of points)
    Let $X$ be a variety over $k$. The Hilbert scheme of points on $X$ is defined as
    \[
    \Hilbl(X) = \{ [Y] \mid Y \subset X, \ \dim(Y) = 0 \ \text{and} \ \text{length}(Y) = l \}.
    \]
where length($Y) = \sum_{y \in Y}{\dim_{k}(\mathcal{O}_{Y,y})}$.
\end{definition}

\begin{definition} \cite{MR2371378}
    Let $X$ be a variety of dimension $d$ over $k$. For a point $x \in X$, the $n$-th infinitesimal neighborhood of $x$ is $x^{(n)} := \operatorname{Spec}\mathcal{O}_{X,x}/m_x^{n+1}$. If $x$ is a smooth point, then $[x^{(n)}] \in \Hilb(X)$. This gives us a morphism
\begin{align*}
    \sigma_n :
    X_{\operatorname{sm}} \longrightarrow \Hilb(X),
    \quad
    x \longmapsto [x^{(n)}].
\end{align*}
The Nash blow-up of order $n$ of $X$, denoted
$\Nashn(X)$, is defined to be the closure of the graph
$\Gamma_{\sigma_n}$ inside $X \times_k \Hilb(X)$. It comes equipped with the natural projection
\begin{align*}
    \pi_n: \Nashn(X) \longrightarrow X
\end{align*}
\end{definition}
Higher Nash blow-ups and modules of principal parts are related via the following classical construction.
\begin{definition} \cite{MR0238860}
    Let $X$ be reduced Noetherian scheme and $\mathcal{M}$ a coherent
$\mathcal{O}_{X}$-module locally free of constant rank $r$ on an open dense
subscheme $U \subseteq X$. The Nash blowup of $X$ associated
to $\mathcal{M}$, denoted $\Nash(X,\mathcal{M})$, is defined to be the closure of the graph of the natural map $U \longrightarrow \Grass_{r}(\mathcal{M})$.
\end{definition}
The next result relates the above two definitions.
\begin{prop} \cite{MR2371378}
    Let $X$ be a variety of dimension $d$, then for every $n \in \mathbb{N}$, we have
    \[
    \Nashn(X) \cong \Nash(X,P_{X}^n) \cong \Nash(X,P_{X,+}^n)
    \]
\end{prop}
Notice that when $n=1$, $\NashO(X)$ is isomorphic to the classical Nash blow-up of $X$. 

\begin{definition}[Torsion submodule]
    Let $R$ be a Noetherian ring and $M$ be a $R$-module. We define torsion submodule of $M$, denoted $\tors_R(M)$, as
    \begin{align*}
        \tors_R(M) = \{m \in M \mid \exists \ \text{a non-zero divisor} \ r \in R \ \text{such that} \ rm = 0\}.
    \end{align*}
From now on, we will denote
\begin{align*}
    \widetilde{M} = M/\tors_R(M).
\end{align*}
This definition extend naturally to the global case. For a quasi-coherent module $\mathcal{M}$ on $X$, we write
\begin{align*}
    \widetilde{\mathcal{M}} = \mathcal{M}/\tors_R(\mathcal{M}).
\end{align*}
\end{definition}

We now give the universal property of higher Nash blow-up.
\begin{prop} \cite{MR2371378,MR1218672} \label{universal}
    Let $X$ be a variety of dimension $d$ over $k$. Let $n \in \mathbb{N}$, and we have the natural projection map
    \begin{align*}
        \pi_n: \Nashn(X) \longrightarrow X.
    \end{align*}
    Suppose we also have a morphism $f: Y \longrightarrow X$ such that $\widetilde{f^*(\mathcal{P}_{X,+}^n)}$ is locally free, then there exists a unique morphism $g: Y \longrightarrow X$ such that $\pi_n\circ g = f$. Equivalently, one can use $\widetilde{f^*(\mathcal{P}_{X,+}^n)}$ instead of $f^*(\widetilde{\mathcal{P}_{X,+}^n})$.
\end{prop}

\section{Reducing higher Nobile's theorem to classical Nobile's theorem}

We will use the following setup for the rest of this paper.
\begin{setup}\label{setup}
We use the following notations.
\begin{itemize}
    \item We fix an affine algebraic vaiety $X \subset \mathbb{A}^\text{s}$ of dimension $d$ over the base field $k$.
    \item We write $\mathbb{A}^\text{s} = \operatorname{Spec}S$ and $X = \operatorname{Spec}B$ where $B = S/I$.
    \item We fix a closed point $x \in X$, and set $R: = B_{m_x} = \mathcal{O}_{X,x}$ where $m_x$ is the maximal ideal corresponding to $x$.
    \item Let $K$ be the field of fractions of $R$.
\end{itemize}
\end{setup}

To prove the generalized Nobile's theorem we need a key property. We define this property as follows:
\begin{definition}[$\mathrm{HN}_n$]\label{HNn}
    Let $R$ be a $k$-algebra, $m_R$ be a maximal ideal and $\dim R = d$. Given $n \in \mathbb{N}$, we say that $R$ satisfies the $\text{HN}_n$ property if there exists algebraically independent elements $x_1, \ldots, x_d \in m_R$, corresponding to a ring homomorphism $A:= k[x_1, \ldots, x_d] \longrightarrow R$, such that the induced map
\begin{align*}
    P_{A,+}^n \otimes_A R/m_R \longrightarrow \widetilde{P}_{R,+}^{n} \otimes_R R/m_R
\end{align*}
is injective. Notice that this is equivalent to saying that
\begin{align*}
    P_{A}^n \otimes_A R/m_R \longrightarrow \widetilde{P}_{R}^{n} \otimes_R R/m_R
\end{align*}
is injective.
\end{definition}

We now see how higher Nobile's theorem reduces down to the classical Nobile's theorem.

\begin{thm}\label{NashntoNashO}
    Let $X$ be an affine algebraic variety, and $R$ be as in Setup~\ref{setup}. Suppose $R$ satisfies the \nameref{HNn} property, and $\Nashn(X) \cong X$ for some $n \in \mathbb{N}$, then $\NashO(X) \cong X$
\end{thm}

\begin{proof}
    We assume $\Nashn(X) \cong X$. Then by the definition of $n$-th Nash blow-up we have that $\widetilde{\pi_{n}^*(\mathcal{P}_{X,+}^n)}$ is locally free of rank $D-1 := \binom{n+d}{d}-1$. But $\pi_n$ is an isomorphism, so $\widetilde{\mathcal{P}}_{X,+}^n$ is locally free of rank $D-1$ which implies $\widetilde{P}_{\mathcal{O}_{X,x},+}^n$ is free of rank $D-1$ for all $x \in X$. Let $R := \mathcal{O}_{X,x}$. By Theorem~\ref{HNNgraded}, we have $A = k[x_1, \ldots, x_d]$ such that
    \begin{align*}
        {P_{A}^n} \otimes_A R/m_R \longrightarrow {\widetilde{P}_{R}^{n}} \otimes_R R/m_R,
    \end{align*}
is injective. It is equivalent to say that
\begin{align*}
        {P_{A,+}^n} \otimes_A R/m_R \longrightarrow {\widetilde{P}_{R,+}^{n}} \otimes_R R/m_R,
    \end{align*}
is injective. But $\dim(P_{A,+}^n \otimes_A R/m_R) = \dim(\widetilde{P}_{R,+}^{n} \otimes_R R/m_R) = D-1$, and so the above map is an isomorphism. Thus, $\widetilde{P}_{R,+}^{n} \otimes_R R/m_R$ is generated by $\big\{[{\overline{{d^n}({x_1})}]^{\alpha_1}} \cdots [{\overline{{d^n}({x_d})}]^{\alpha_d}} \mid 1 \leq |\alpha| \leq n\big\}$ as $R/m_R$ vector space. Consider the surjective map
\begin{align*}
    \widetilde{P}_{R}^{n} \longrightarrow \widetilde{P}_{R}^{n} \otimes_R R/m_R \cong \frac{\widetilde{P}_{R}^{n}}{m_R\widetilde{P}_{R}^{n}}.
\end{align*}
Further, it follows from Nakayama's lemma that the $R$-module $\widetilde{P}_{R,+}^{n}$ is freely generated by $\big\{{\overline{{d^n}({x_1})}^{\alpha_1}} \cdots {\overline{{d^n}({x_d})}^{\alpha_d}} \mid 1 \leq |\alpha| \leq n\big\}$. Furthermore, we have a surjective $R$-module homomorphism
\begin{align*}
    P_{R,+}^{n} = I_R/I_R^{n+1} \longrightarrow \Omega_{R} = I_R/I_R^2.
\end{align*} 
This induces the following surjective map
\begin{align*}
    \widetilde{P}_{R,+}^{n} \longrightarrow \widetilde{\Omega}_{R}.
\end{align*}
It follows that the $R$-module $\widetilde{\Omega}_{R}$ is generated by $\big\{\overline{d({x_1})}, \cdots, \overline{d({x_d})}\big\}$. But $\operatorname{rank}(\widetilde{\Omega}_{R}) = \dim R = d$. Thus, $\widetilde{\Omega}_{R} = \Omega_{\mathcal{O}_{X,x}}/\text{tors}_{\mathcal{O}_{X,x}}(\Omega_{\mathcal{O}_{X,x}})$ is free of rank $d$. This is equivalent to $\widetilde{\Omega}_X$ being locally free of rank $d$. So, by the universal property~\ref{universal} of Nash blow-up, there exists a unique map $\varphi : X \longrightarrow \NashO(X)$. Thus, $\pi : \text{Nash}_1(X) \longrightarrow X$ is an isomorphism.
\end{proof}

\section{Arc lemma}
In this section, we construct an arc and observe some of its important properties.

\begin{remark}\label{projection}
    We can assume $x \in X$ the origin of $\mathbb{A^\text{s}}$. Now we consider linear projections of the form $\mathbb{A^\text{s}} \longrightarrow Y$, where $Y$ is an affine subspace of $\mathbb{A^\text{s}}$ of dimension $d$ containing $0$. For each such projection we can choose affine coordinates $x_1, \ldots, x_s$ in $\mathbb{A^\text{s}}$ centered at $0$ such that $Y$ is defined by the equations $x_{d+1} = \ldots = x_s = 0$ and the projection is given by $(x_1, \ldots, x_s) \in \mathbb{A^\text{s}} \mapsto (x_1, \ldots, x_d) \in Y$ (notice that $(x_1, \ldots, x_d)$ give coordinates in $Y$). Let $Z$ be the subspace of $\mathbb{A^\text{s}}$ given by $x_1, \ldots, x_d = 0$. This $Z$ gives the direction of projection. The fibers of the projection are translations of Z, that is, the affine subspaces of the form $y + Z$, where $y \in Y$ and $+$ is the addition for the vector space structure of $\mathbb{A^\text{s}}$ induced by the coordinates $x_1, \ldots, x_s$. Each such projection induces a map $X \longrightarrow Y$, and let $\mathcal{H}$ be the set of all such induced projections. When this induced map is finite we get a Noether normalization which means that $k[x_1, \ldots, x_d] \longrightarrow B$ is a finite integral extension. Then the proof of the Noether normalization lemma shows that the subset $\mathcal{U} \subseteq \mathcal{H}$ of the projections that give a Noether normalization is an open dense subset of $\mathcal{H}$.
\end{remark}

Firstly, we give the graded arc lemma.
\begin{lemma}[Graded arc lemma]\label{gradedarc}
Let $X$ be a graded algebraic variety, and $R$ be as in Setup~\ref{setup}. Then for a generic choice of coordinates $x_1, \ldots, x_s$ for $\mathbb{A^\text{s}}$, we get $\{x_1, \ldots, x_d\}$ a transcendental basis of $K$ over $k$. Further, there also exists a field $L$, and an arc $\beta: R \longrightarrow L[[t]]$ such that $\beta$ is an inclusion and $\beta(x_i) = {u_i}{t}$, where $u_i \in L$ for all $1 \leq i \leq s$, and $\{u_1, \ldots, u_d\}$ is a transcendental basis of $L$ over $k$.
\end{lemma}
\begin{proof}
   From Setup~\ref{setup} before the theorem, we know $B$ is a finitely generated $k$-algebra. By the above Remark~\ref{projection}, we can pick a projection and coordinates $x_1, \ldots, x_s$ for $\mathbb{A^\text{s}}$. We can assume that $\{x_1, \ldots, x_d\}$ is a transcendental basis of $K$ over $k$ such that $A \subseteq B$, where $A = k[x_1, \ldots, x_d]$ is a Noether normalization. Since $B$ is a domain, we have $B \subseteq R$, and so $A \subseteq R$. Let $u_1, \ldots, u_d$ be new transcendental elements, and let $L$ be the algebraic closure of the field of fractions of $k[u_1, \ldots, u_d]$.
 
 Consider the generic point of $X$, and since $R$ is integral over $A$, we have a map $\varphi: R \longrightarrow L$ such that $\varphi(x_i) = u_i$ for all $1 \leq i \leq d$. We can now scale the generic point. More precisely, since $X$ is graded, we get an arc $\widetilde{\beta}: S \longrightarrow L[[t]]$ via $\widetilde{\beta}(x_i) = {u_i}{t}$, for all $1 \leq i \leq s$. Then, since the ideal $I$ is homogeneous, $\widetilde{\beta}$ factors via a ring map $\beta: R \longrightarrow L[[t]]$ still satisfying $\beta(x_i) = u_it$ for all $1 \leq i \leq s$.
\end{proof}

For general rings, we have the following version.

\begin{lemma}[Arc lemma]\label{arc}
Let $X$ and $R$ be as in Setup~\ref{setup}, and assume $\text{char}(k) = 0$. Then for a generic choice of coordinates $x_1, \ldots, x_s$ for $\mathbb{A^\text{s}}$, we get $\{x_1, \ldots, x_d\}$ a transcendental basis of $K$ over $k$. Further, there also exists a field $L$, and an arc $\beta: R \longrightarrow L[[t]]$ such that $\beta$ is an inclusion, and
\begin{enumerate}[label=(\alph*)] 
    \item $\beta(x_i) = {u_i}{t^a}$, where $u_i \in L$ for all $1 \leq i \leq d$, and and $\{u_1, \ldots, u_d\}$ is a transcendental basis of $L$ over $k$.
    \item $\beta(x_j) = {u_j}{t^{b_j}}$, where $u_j \in L[[t]]$ and $b_j \geq a$ for all $d+1 \leq j \leq s$. 
\end{enumerate}
\end{lemma}

\begin{proof}
We now start by proving statement (a) of the theorem. Again by the Remark~\ref{projection}, we can pick a projection and coordinates $x_1, \ldots, x_s$ for $\mathbb{A^\text{s}}$. We can assume that $\{x_1, \ldots, x_d\}$ is a transcendental basis of $K$ over $k$ such that $A \subseteq B$, where $A = k[x_1, \ldots, x_d]$ is aNoether normalization. Since $B$ is a domain, we have $B \subseteq R$, and so $A \subseteq R$. Let $u_1, \ldots, u_d$ be new transcendental elements, and let $L$ be the algebraic closure of the field of fractions of $k[u_1, \ldots, u_d]$. Then consider the arc $\beta_0: A \longrightarrow L[[s]]$ such that ${\beta_0}(x_i) = {u_i}{s}$. Notice that $\beta_0$ is clearly an inclusion. Further, let $K$ and $K_0$ be the fields of fractions of $R$ and $A$ respectively. Then $K$ is an algebraic extension of $K_0$. We denote by $K_0^{\text{ac}}$ the algebraic closure of $K_0$, so we have $K_0 \subseteq K \subseteq K_0^{\text{ac}}$. Next we denote by $L((s))$ the field of fractions of $L[[s]]$, and we let $L((s))^{\text{ac}}$ be the algebraic closure of $L((s))$. Let $\delta_0: K_0 \longrightarrow L((s))$ be the map induced by $\beta_0$. Then notice that the morphism $\delta^{\text{ac}}: K_0^{\text{ac}} \longrightarrow L((s))^{\text{ac}}$ exists by the definition of algebraic closure. Since $L$ is an algebraically closed field of characteristic zero and by the Newton-Puiseux theorem, we have $L((s))^{\text{ac}} = \bigcup_{a \geq 1} L((s^{1/a}))$. Furthermore, we know that $K$ is a finite field extension of $K_0$ since $R$ is a finite ring extension of $A$. By looking at the images via $\delta^{\text{ac}}$ of each of the finitely many generators of $K$ over $K_0$, we see that $\delta^{\text{ac}}(K)$ lies inside of $L((s^{1/a}))$ for some non-negative integer $a$. This gives the map $\beta: K \longrightarrow L((s^{1/a}))$ and the following commutative diagram:
\[
    \begin{tikzcd}[row sep=small]
        K_0 \arrow{r}{\delta_0} \arrow[swap]{d}
        & L((s)) \arrow{d}
    \\
        K \arrow{r}{\delta} \arrow[swap]{d}
        & L((s^{1/a})) \arrow{d}
    \\
        K_0^{\text{ac}} \arrow{r}{\delta^{\text{ac}}}
        & L((s))^{\text{ac}}.
    \end{tikzcd}
\]
Next notice that $L[[s^{1/a}]]$ is the integral closure of $L[[s]]$ inside of $L((s^{1/a}))$. Since $R$ is integral over $A$, we see that $\delta(R) \subseteq L[[s^{1/a}]]$ giving the map $\beta: R \longrightarrow L[[s^{1/a}]]$ and the following commutative diagram:
\[ 
    \begin{tikzcd}[row sep=small]
        A \arrow{r}{\beta_0} \arrow[swap]{d}
        & L[[s]] \arrow{d}
    \\
        R \arrow{r}{\beta}
        & L[[s^{1/a}]].
    \end{tikzcd}
\]
Also, by the construction, $\beta$ is an inclusion. Let $t = s^{1/a}$ and then we have constructed the arc $\beta: R \longrightarrow L[[t]]$ such that $\beta(x_i) = {u_i}{t^a}$ where $u_i \in L$ for all $1 \leq i \leq d$.
    
The next step is to prove the part (b) of the theorem. Consider the maximal ideal $m_x \subset \mathcal{O}_X$ and let $b \geq 1$ be order of contact of $a$ with $m_x$ which means $\beta^{-1}(m_x)  \cdot L[[t]] = (t^b)L[[t]]$. Here, $\beta$, $L[[t]]$ and $a$ are from part (a). Notice that $m_x$ is generated by $x_1, \ldots, x_s$. Therefore, for each $\{1, \ldots, s\}$ we have $\beta(x_i) = t^{b_i}.{u_i}(t)$ with ${u_i}(t) \in L[[t]] \mysetminus (t)$, $b_i \in \mathbb{Z}_{\geq 1} \cup \{\infty\}$ and $b = \text{min}\{b_1, \ldots, b_s\} \in \mathbb{Z}_{\geq 1}$. We see that $a = b_1 = \ldots = b_d \geq b$. Here, we use the convention $t^{\infty} = 0$.

Let $C \in \mathbb{A^\text{s}}$ be the tangent cone to $X$ at $x$. The ideal of $C$ is $I_C = \operatorname{in}(I) = \{\text{in}(f) : f \in I\}$ where $\operatorname{in}(f)$ denotes the lowest homogeneous part of the polynomial $f$. Since $\beta$ is an arc on $X$, we know that $\beta(f) = 0$ for $f \in I$. For $f \in I$, write $f = f_e + f_{e+1} + \cdots$, where $e \geq 1$ and $f_j$ is homogeneous of degree $j$. Write $\beta(x_i) = {v_i}t^b +$ (h.o.t) for $i \in \{1, \ldots, s\}$, where (h.o.t) means higher order terms. Let $v = (v_1, \ldots, v_s)$. Then we see that $0 = \beta(f) = {f_e}(v)t^{eb} + \text{(h.o.t)} = \text{in}(f)(v)t^{eb}$ + (h.o.t). Therefore $\text{in}(f)(v) = 0$ and thus $v \in C$. We also know that $v \neq x$, and it follows that $v \in C \mysetminus \{0\}$.

Recall that $Z$ is given by $x_1 = \cdots = x_d = 0$. Using the notation of Remark$~\ref{projection}$, let $\mathcal{V} \subseteq \mathcal{H}$ be the set of all projections for which $Z$ intersects $C$ equally at the origin, i.e. $Z \cap C = \{0\}$. Then $\mathcal{V}$ is a dense open subset of $\mathcal{H}$. Thus, pick any projection in $\mathcal{U} \cap \mathcal{V}$. Notice that $v_i = {u_i}(0)$ if $b_i = b$ and $v_i = 0$ if $b_i > b$. Then $v \notin Z$ since $v \neq x$ and $v \in C$. Therefore, $a = b$ and $b_j \geq a$ for all $d+1 \leq j \leq s$.
\end{proof}




\section{Construction}\label{construction}
In this section, we are going to construct the following commutative diagram which will be essential in proving the \nameref{HNn} property.
\begin{equation}\label{diagram}
    \begin{tikzcd}
P_{A,+}^n \otimes_A R \arrow[r,"{d{\varphi}}"] \arrow[d] & \widetilde{P}_{R,+}^{n} \arrow[r, "{d{\beta}}"] \arrow[d] & \widehat{P}_{L[[t]]/L,+}^n \arrow[d] \\
{P_{A,+}^n \otimes_A R/m_R} \arrow[r,"{\overline{d{\varphi}}}"] & {\widetilde{P}_{R,+}^{n} \otimes_R R/m_R} \arrow[r, "\overline{{d{\beta}}}"] & {\widehat{P}_{L[[t]]/L,+}^n}/{N_L}.
 \end{tikzcd}
\end{equation}

Firstly, let $\varphi: A \longrightarrow R$ be the Noether normalization map and $\beta: R \longrightarrow L[[t]]$ be the arc constructed in the graded arc lemma~\ref{gradedarc} or the arc lemma~\ref{arc}. Recall that in the graded case, $k$ has arbitrary characteristic; in the non-graded case, we need to assume $\operatorname{char}(k) = 0$.

We now consider the following composition:
\[ \begin{tikzcd}
A \arrow{r}{\varphi} & R \arrow{r}{\beta} & L[[t]].
\end{tikzcd}
\]
This induces the following composition of $R$-module homomorphisms for all $n \in \mathbb{N}$
\[ \begin{tikzcd}
R \otimes_A P_{A,+}^n \arrow{r}{d{\varphi}} & P_{R,+}^{n} \arrow{r}{d{\beta}} & P_{L[[t]]/L,+}^n.
\end{tikzcd}
\]
We also have the ring morphism $P_{L[[t]]/L}^n \longrightarrow \widehat{P}_{L[[t]]/L}^n$ where $\widehat{P}_{L[[t]]/L}^n$ is the $(t)$-adic completion of the $L[[t]]$-module $P_{L[[t]]/L}^n$. We know $P_{L[[t]]/L}^n \cong P_{L[[t]]/L,+}^n \oplus L[[t]]$, and so we get that $\widehat{P}_{L[[t]]/L}^n \cong \widehat{P}_{L[[t]]/L,+}^n \oplus L[[t]]$. Thus, one gets a morphism $P_{L[[t]]/L,+}^n \longrightarrow \widehat{P}_{L[[t]]/L,+}^n$. This gives us
\[ \begin{tikzcd}
R \otimes_A P_{A,+}^n \arrow{r}{d{\varphi}} & P_{R,+}^{n} \arrow{r}{d{\beta}} & \widehat{P}_{L[[t]]/L,+}^n.
\end{tikzcd}
\]
We work with $\widehat{P}_{L[[t]]/L,+}^n$ because it is free of rank $n$ with basis $\{d^nt, (d^nt)^2, \cdots{, (d^nt)^n}\}$. On the other hand, $P_{L[[t]]/L,+}^n$ is not finitely generated.
\begin{lemma}
    Let $\tors_R(P_{R,+}^{n})$ be the torsion $R$-submodule of $P_{R,+}^{n}$. Then we have that $\tors_R(P_{R,+}^{n}) \subseteq \operatorname{ker}(d{\beta})$.
\end{lemma}

\begin{proof}
    Let $\omega \in \tors_R(P_{R,+}^{n})$. Then there exists a non-zero $f \in R$ such that $f\omega = 0$. This implies $d{\beta}(f\omega) = 0$. But $d{\beta}(f\omega) = \beta(f)d{\beta}(\omega)$, and so $\beta(f)d{\beta}(\omega) = 0$. We know that $\beta$ is an inclusion. Thus, $\beta(f) \neq 0$. Also, $\widehat{P}_{L[[t]]/L,+}^n$ is free and hence torsion-free. Then we get $d{\beta}(\omega) = 0$. Therefore, $\tors_R(P_{R,+}^{n}) \subseteq \operatorname{ker}(d{\beta})$.
\end{proof}
Then from the above lemma, $d{\beta}$ factors through $\widetilde{P}_{R,+}^{n}$. We then have the following commutative diagram:

\[ 
\begin{tikzcd}
P_{A,+}^n \otimes_A R \arrow[r,"{d{\varphi}}"] \arrow[dr] & P_{R,+}^{n} \arrow[r, "{d{\beta}}"] \arrow[d, "\pi"] & \widehat{P}_{L[[t]]/L,+}^n \\
& \widetilde{P}_{R,+}^{n} \arrow[ur]
\end{tikzcd}.
\]

Consider the quotient map $\pi: P_{R,+}^{n} \longrightarrow \widetilde{P}_{R,+}^{n}$. We know that $P_{R,+}^{n}$ is generated as an $R$-module by $\{{{d^n}({x_1})^{\alpha_1}} \cdots {{d^n}({x_s})^{\alpha_s}} | 1 \leq |\alpha| \leq n\}$, where $\alpha = (\alpha_1, \ldots, \alpha_S)$ denotes a multi-index, and $|\alpha| = \alpha_1 + \cdots + \alpha_s$. Then $\widetilde{P}_{R,+}^{n}$ is generated by $\{\pi({{d^n}({x_1})^{\alpha_1}} \cdots {{d^n}({x_s})^{\alpha_s}}) \mid 1 \leq |\alpha| \leq n\}$. We will do abuse of notation and say $\widetilde{P}_{R,+}^{n}$ is generated as an $R$-module by $\{{{d^n}({x_1})^{\alpha_1}} \cdots {{d^n}({x_s})^{\alpha_s}} \mid 1 \leq |\alpha| \leq n\}$, and we write

\[ \begin{tikzcd}
P_{A,+}^n \otimes_A R \arrow{r}{d{\varphi}} & \widetilde{P}_{R,+}^{n} \arrow{r}{d{\beta}} & \widehat{P}_{L[[t]]/L,+}^n.
\end{tikzcd}
\]

Denote $d{\beta_A}:= d{\beta} \circ d{\varphi}: P_{A,+}^n \otimes_A R \longrightarrow \widehat{P}_{L[[t]]/L,+}^n$ which is the induced map from the arc $\beta \circ \varphi: A \longrightarrow L[[t]]$. 

Consider the $L[[t]]$-submodule of $\widehat{P}_{L[[t]]/L,+}^n$, denoted $N_L$, defined as
\begin{align}\label{N_L}
  N_L :=  \operatorname{span}\{(t^a)d{\beta}({{d^n}({x_1})^{\alpha_1}} \cdots {{d^n}({x_s})^{\alpha_s}}) \mid 1 \leq |\alpha| \leq n\},
\end{align}
where $a$ is the non-negative integer from Lemma~\ref{gradedarc} or Lemma~\ref{arc}. Then we have that $d{\varphi}({m_A}(P_{A,+}^n \otimes_A R) )\subseteq d{\beta}({m_R}\widetilde{P}_{R,+}^{n}) \subseteq N_L$ where $m_A:= (x_1, \ldots, x_d)A$ and $m_R:= (x_1, \ldots, x_s)R$ are the maximal ideals of $A$ and $R$ respectively. To see this, let $f \in m_R$ and $\omega \in \widetilde{P}_{R,+}^{n}$, then
\begin{align*}
    d\beta(f\omega) &= \beta(f)d\beta(\omega) \\ &= wt^bd\beta(\omega),
\end{align*}
where $b \geq a$ and $w$ is a unit. Thus, we get the required commutative diagram~(\ref{diagram}).

\begin{remark}\label{calculation}
    Using Remark~\ref{HigherLeibniz} with induction, we can see that
\[
d{\beta}_A({{d^n}({x_1})^{\alpha_1}} \cdots {{d^n}({x_d})^{\alpha_d}}) = {{d^n}({{u_1}{t^a}})^{\alpha_1}} \cdots {{d^n}({{u_d}{t^a}})^{\alpha_d}},
\]
where ${d^n}: L[[t]] \longrightarrow P_{L[[t]]/L,+}^n$ is the universal $n$-th order differential over $L[[t]]$ and $1 \leq |\alpha| \leq n$. We do it for the basic case and it follows easily in the general case by applying induction.
\begin{align*}
    d{\beta}_A({{d^n}({x_i}){d^n}({x_j})}) &= d{\beta}_A({d^n}(x_ix_j) - x_i{d^n}(x_j) - x_j{d^n}(x_i)) \\ &= d{\beta}_A({{d^n}(x_ix_j))} - d{\beta}_A(x_i{{d^n}(x_j))} - d{\beta}_A(x_j{{d^n}(x_i))} \\ &= {d^n}(u_itu_jt) - u_it{d^n}(u_jt) - u_jt{d^n}(u_it) \\ &= {d^n}(u_it){d^n}(u_jt). 
\end{align*}
\end{remark}
These formulas will be useful later to prove the ~\nameref{HNn} property.

\begin{lemma}\label{constlemma}
    If ${\overline{d{\beta}}_A}:= {\overline{d{\beta}} \circ \overline{d{\varphi}}}: P_{A,+}^n \otimes_A R/m_R \longrightarrow {\widehat{P}_{L[[t]]/L,+}^n}/{N_L}$ is injective, then $R$ satisfies the ~\nameref{HNn} property.
\end{lemma}

\begin{proof}
    It is clear from the construction of the commutative diagram~(\ref{diagram}).
\end{proof}

\section{\texorpdfstring{$\operatorname{HN}_n$}{n} property for the graded case}

\begin{thm}\label{HNNgraded}
Let $X$ be a graded algebraic variety, and $R$ be as in Setup~\ref{setup}. Then $R$ satisfies the \nameref{HNn} property.
\end{thm}

\begin{proof}
From Lemma~\ref{constlemma}, it is enough to show that $\overline{d{\beta}}_A$ is injective. Consider the $R$-module homomorphism $d{\beta}: \widetilde{P}_{R,+}^{n} \longrightarrow \widehat{P}_{L[[t]]/L,+}^n$. Since $R$ is a graded ring, then by the graded arc lemma~\ref{gradedarc}, $a = b_{d+1} = \cdots = b_s = 1$. Then from the Remark~\ref{calculation}, we get
\begin{equation*}
    d{\beta}_A({{d^n}({x_1})^{\alpha_1}} \cdots {{d^n}({x_d})^{\alpha_d}}) = {{u_1}^{\alpha_1}\cdots}{u_d}^{\alpha_d}({d^n}t)^{|\alpha|},
\end{equation*}
where $1 \leq |\alpha| \leq n$. We look at $P_{A,+}^n \otimes_A R/m_R$ as a vector space over $k$. Let $\overline{\omega} \in \operatorname{ker}(\overline{d\beta_A})$, and we write
\[
\overline{\omega} = \sum_{1 \leq |\alpha| \leq n}p_{{\alpha_1}\ldots{\alpha_d}}{\overline{{d^n}({x_1})}^{\alpha_1}} \cdots {\overline{{d^n}({x_d})}^{\alpha_d}},
\]
where $p_{{\alpha_1}\ldots{\alpha_d}} \in k$ for all $1 \leq |\alpha| \leq n$. Then $\overline{d{\beta}_A}(\overline{\omega}) = 0$ and we get
\begin{equation*}
    \begin{aligned}
        \overline{d{\beta}_A}(\overline{\omega}) = \overline{d{\beta}_A}\Bigg(\sum_{1 \leq |\alpha| \leq n}p_{{\alpha_1}\ldots{\alpha_d}}{\overline{{d^n}({x_1})}^{\alpha_1}} \cdots {\overline{{d^n}({x_d})}^{\alpha_d}}\Bigg)
    \end{aligned}
\end{equation*}
Let $\omega = \sum_{1 \leq |\alpha| \leq n}p_{{\alpha_1}\ldots{\alpha_d}}{{d^n}({x_1})^{\alpha_1}} \cdots {{d^n}({x_d})^{\alpha_d}} \in P_{A,+}^n \otimes_A R$ be the natural pre-image of $\overline{\omega}$. Then we get
\begin{equation*}
    \begin{aligned}
      {d{\beta}_A}(\omega) = & \Bigg(\sum_{|\alpha| = 1}(p_{{\alpha_1}\ldots{\alpha_d}}{u_1}^{\alpha_1}\cdots{{u_d}^{\alpha_d}})({d^n}t) + \sum_{|\alpha| = 2}(p_{{\alpha_1}\ldots{\alpha_d}}{u_1}^{\alpha_1}\cdots{{u_d}^{\alpha_d}})({d^n}t)^2 + \\ & + \cdots \cdots + \\ & \sum_{|\alpha| = n}(p_{{\alpha_1}\ldots{\alpha_d}}{u_1}^{\alpha_1}\cdots{{u_d}^{\alpha_d}})({d^n}t)^n\Bigg) \in N_L     
    \end{aligned}
\end{equation*} 
Recall from~(\ref{N_L}) that $N_L$ is generated by 
\[
\Big\{(t^a){{d^n}({{u_1}{t^a}})^{\alpha_1}} \cdots {{d^n}({{u_d}{t^a}})^{\alpha_d}}{{d^n}({{v_{d+1}}{t^{b_{d+1}}}})^{\alpha_{d+1}}} \cdots {{d^n}({{v_s}{t^{b_s}}})^{\alpha_s}} \,\Big|\, 1 \leq |\alpha| \leq n\Big\}.
\]
 Notice that the coefficient of $({d^n}t)^i$ in ${d{\beta}_A}(\omega)$ does not involve a factor of $t$. But every element of $N_L$ has a factor of $t$. Thus,
\[
\sum_{|\alpha| = 1}(p_{{\alpha_1}\ldots{\alpha_d}}{u_1}^{\alpha_1}\cdots{{u_d}^{\alpha_d}}) = \cdots = \sum_{|\alpha| = n}(p_{{\alpha_1}\ldots{\alpha_d}}{u_1}^{\alpha_1}\cdots{{u_d}^{\alpha_d}}) = 0.
\]
But $u_1, \ldots, u_d$ are algebraically independent over $k$. Thus, $p_{{\alpha_1}\ldots{\alpha_d}} = 0$ for all $1 \leq |\alpha| \leq n$. This implies $\overline{\omega} = 0$, and so $\overline{d{\beta}}_A$ is injective.
\end{proof}

\begin{remark}\label{graded}
  Let $R$ be a graded ring as above. We know $P_{A}^n \otimes_A R \cong R \oplus (P_{A,+}^n \otimes_A R)$ which gives us $P_{A}^n \otimes_A R/m_R \cong R/m_R \oplus (P_{A,+}^n \otimes_A R/m_R)$. We also have $P_{R/k}^n \cong R \oplus P_{R,+}^{n}$, and it follows that $\widetilde{P}_{R}^{n} \cong R \oplus \widetilde{P}_{R,+}^{n}$. This implies that $\widetilde{P}_{R}^{n} \otimes_R R/m_R \cong R/m_R \oplus (\widetilde{P}_{R,+}^{n} \otimes_R R/m_R)$. Then from the above proposition, it follows that
  \begin{align*}
    \overline{d{\varphi}}: P_{A}^n \otimes_A R/m_R \longrightarrow \widetilde{P}_{R}^{n} \otimes_R R/m_R,
\end{align*}
is injective.
\end{remark}

\section{\texorpdfstring{$\operatorname{HN}_2$}{2} property for \texorpdfstring{$R$}{R}}
In this section, we prove the $\operatorname{HN}_2$ property for $R$.

\begin{thm}\label{HNN_n=2}
Let $R$ be as in Setup~\ref{setup} and assume $\operatorname{char}(k) = 0$. Then 
\begin{align*}
    {\overline{d{\beta}}_A}:= {\overline{d{\beta}} \circ \overline{d{\varphi}}}: P_{A,+}^2 \otimes_A R/m_R \longrightarrow {\widehat{P}_{L[[t]]/L,+}^2}/{N_L},
\end{align*}
in~(\ref{diagram}) is injective. In particular, $R$ satisfies the \nameref{HNn} property when $n = 2$.
\end{thm}

\begin{proof}
    Consider the $R$-module homomorphism $d{\beta}_A: \widetilde{P}_{R,+}^{2} \longrightarrow \widehat{P}_{L[[t]]/L,+}^2$. If $a = 1$, then we are done by the proof of the Theorem~\ref{HNNgraded}. So, we assume $a \geq 2$. Then from Remark~\ref{calculation}, and by Remark~\ref{HigherLeibniz} we get
\begin{equation*}
    d{\beta}_A({d^2}({x_i})) = (a{u_i}t^{a-1})({d^2}t) + \bigg(\binom{a}{2}{u_i}t^{a-2}\bigg)({d^2}t)^2,
\end{equation*}
and
\begin{equation*}
    d{\beta}_A(({d^2}{(x_i}))({d^2}({x_j})) = (a^2u_iu_jt^{2a-2})({d^2}t)^2.
\end{equation*}
Let $\overline{\omega} \in \operatorname{ker}(\overline{d\beta_A})$, and we write
\[
\overline{\omega} = \sum_{1 \leq i \leq d}p_i{\overline{{d^2}({x_i})}} + \sum_{1 \leq j \leq l \leq d}p_{jl}{\overline{{d^2}({x_j}){d^2}({x_l})}},
\]
where $p_i, p_{jl} \in k$ for all $1 \leq i \leq d$ and $1 \leq j \leq l \leq d$. Let $\omega = \sum_{1 \leq i \leq d}p_i{{d^2}({x_i})} + \sum_{1 \leq j \leq l \leq d}p_{jl}{({x_j}){d^2}({x_l})} \in P_{A,+}^2 \otimes_A R$ be the natural pre-image of $\overline{\omega}$. Then we have $d{\beta}_A(\omega) \in N_L$ where 
\begin{equation*}
    d{\beta}_A(\omega) = \sum_{1 \leq i \leq d}(ap_iu_it^{a-1})({d^2}t) + \bigg(\sum_{1 \leq j \leq l \leq d}p_i\binom{a}{2}{u_i}t^{a-2} + \sum_{1 \leq j \leq l \leq d}a^2p_{jl}u_ju_lt^{2a-2}\bigg)({d^2}t)^2.
\end{equation*}
Recall from~(\ref{N_L}) that $N_L$ is generated by 
\[
\Big\{(t^a){{d^2}({{u_1}{t^a}})^{\alpha_1}} \cdots {{d^2}({{u_d}{t^a}})^{\alpha_d}}{{d^2}({{v_{d+1}}{t^{b_{d+1}}}})^{\alpha_{d+1}}} \cdots {{d^2}({{v_s}{t^{b_s}}})^{\alpha_s}} \,\Big|\, 1 \leq |\alpha| \leq 2\Big\}.
\]
Notice that for any element of $N_L$, the coeffecient of ${d^2}t$ has a factor of $t^{2a-1}$. Since $d{\beta}_A(\omega) \in N_L$, we get
\begin{equation*}
    \sum_{1 \leq i \leq d}ap_iu_i = 0.
\end{equation*}
But $u_1, \ldots, u_d$ are algebraically independent over $k$. Since $\operatorname{char}(k) = 0$, $p_i = 0$ for all $1 \leq i \leq d$. Further, we reduce down to
\begin{equation*}
    d{\beta}_A(\omega) = \bigg(\sum_{1 \leq j \leq l \leq d}a^2p_{jl}u_ju_lt^{2a-2}\bigg)({d^2}t)^2 \in N_L.
\end{equation*}
Then the only possibility is
\begin{align*}
    \bigg(\sum_{1 \leq j \leq l \leq d}a^2p_{jl}u_ju_lt^{2a-2}\bigg)({d^2}t)^2 &= \sum_{1 \leq i \leq d}g_it^ad{\beta}_A({d^2}({x_i})) \\ &= \sum_{1 \leq i \leq d}\bigg((g_ia{u_i}t^{2a-1})({d^2}t) + \bigg(g_i\binom{a}{2}{u_i}t^{2a-2}\bigg)({d^2}t)^2\bigg).
\end{align*}
where $g_i \in L[[t]]$. Further, comparing the coefficients of ${d^2}t$ on both sides, we get
\begin{align*}
    \sum_{1 \leq i \leq d}g_ia{u_i} = a\sum_{1 \leq i \leq d}g_i{u_i} = 0.
\end{align*}
And, so
\begin{align*}
    \sum_{1 \leq i \leq d}g_i\binom{a}{2}{u_i} = \binom{a}{2}\bigg(\sum_{1 \leq i \leq d}g_i{u_i}\bigg) = 0.
\end{align*}
It follows that
\begin{align*}
    \sum_{1 \leq j \leq l \leq d}a^2p_{jl}u_ju_l = 0.
\end{align*}
But again we use that $u_1, \ldots, u_d$ are algebraically independent over $k$. Again since $\operatorname{char}(k) = 0$, $p_{jl} = 0$ for all $1 \leq j \leq l \leq d$. This implies $\overline{\omega} = 0$, and so $\overline{d{\beta}}_A$ is injective.
\end{proof}


\section{Filtration and associated graded ring of \texorpdfstring{$P_{R}^n$}{PRn}}
We were not able to show the~\nameref{HNn} property for the general case. But in sections 8 and 9, we try to use~\nameref{HNn} property of the graded case and extend it to the general case. As a result of this approach, we get a nice corollary at the end of section 9.

We start by constructing a filtration of $P_{R}^n$ which gives us the associated graded module of $P_{R}^n$. One can then compare it with the module of principal parts of the associated graded ring. For this section, we use the Setup ~\ref{setup} and the following notations.
\begin{itemize}
    \item $(R, m_R)$ be a local ring as before (not necessary homogeneous)
    \item $G := \operatorname{gr}_{m_R}(R)$ - associated graded ring of $R$ (with respect to the $m_R$-adic filtration)
    \item $m_G := \operatorname{gr}_{m_R}^{+}(R)$ - homogeneous maximal ideal
\end{itemize}

\noindent We define the filtration of $P_{R}^{n}$ as follows:

Let $S':= k[x_1, \ldots, x_s, \xi_1, \ldots, \xi_s]$ and $m_{S'} := (x_1, \ldots, x_s, \xi_1, \ldots, \xi_s)$ be its maximal ideal. Define
\begin{align*}
    S_l = \{f \in S' \mid \operatorname{order}(f) \geq l\}
\end{align*}
We can also write it as
\begin{align*}
    S_l = (x_1, \ldots, x_s, \xi_1, \ldots, \xi_s)^l
\end{align*}

It is easy to see that $S_{l_1} \cdot S_{l_2} \subseteq S_{l_1+l_2}$
This defines a grading on $S'$, and we get
\begin{align*}
    \operatorname{gr}(S') = \bigoplus_{l \geq 0} S_l/S_{l+1}.
\end{align*}

Let $R' = S'/I'$ where $I' = (f_1, \ldots, f_c, g_1, \ldots, g_c) + (x_i - \xi_i, \ldots, x_s - \xi_s)^{n+1}$ is an ideal in $S'$. Here, $(f_1, \ldots, f_c)$ is the ideal in $k[x_1, \ldots, x_s]$ defining our variety $X$, and $(g_1, \ldots, g_c)$ is the ideal defining $X$ but with different coordinates. It means that $f_i$ and $g_i$ are exactly the same polynomials for all $1 \leq i \leq c$ which we get by replacing $x_i$ by $\xi_i$ and vice versa. We define the filtration on $R'$ similar to $S'$ by letting
\begin{align*}
    R_l := (S_l + I')/I' \cong S_l/(I' \cap S_l).
\end{align*}
Then it follows that $R_{l_1} \cdot R_{l_2} \subseteq R_{l_1+l_2}$ and we get
\begin{align*}
    \operatorname{gr}(R') = \bigoplus_{l \geq 0} R_l/R_{l+1}.
\end{align*}
We know that $P_{R}^{n} = (R \otimes R)/I_R^{n+1}$ where $R := (k[x_1, \ldots, x_s]/(f_1, \ldots, f_c))_{(x_1, \ldots, x_s)}$.Then, $P_{R}^{n}$ can be identified with $R'_{(x_1, \ldots, x_s, \xi_1, \ldots, \xi_s)}$, i.e, 
\begin{align}\label{eq}
    P_{R}^{n} = \frac{k[x_1, \ldots, x_s, \xi_1, \ldots, \xi_s]_{(x_1, \ldots, x_s, \xi_1, \ldots, \xi_s)}}{(f_1, \ldots, f_c, g_1, \ldots, g_c) + (x_1 - \xi_1, \ldots, x_s - \xi_s)^{n+1}}.
\end{align}
We also have the initial map with the above filtration on $S$, denoted, 
\begin{align} \label{In}
    \text{In}: S' \longrightarrow \operatorname{gr}(S').
\end{align}
This map gives the homogeneous part of lowest degree of a polynomial.
The above defined filtration is $m_R$-compatible. It tells us that the associated graded ring of a localization of a ring is the same as the associated graded ring of the ring, and so we get
\begin{align*}
    \operatorname{gr}(P_{R}^{n}) = \frac{k[x_1, \ldots, x_s, \xi_1, \ldots, \xi_s]}{\text{In}((f_1, \ldots, f_c, g_1, \ldots, g_c) + (x_1 - \xi_1, \ldots, x_s - \xi_s)^{n+1})}.
\end{align*}
\begin{definition}
    Let $I = (h_1, \ldots, h_l)$ be an ideal. We say $\{h_1, \ldots, h_l\}$ is a Gr{\"o}bner basis with respect to $\text{In}$ if $\text{In}(I) = (\text{In}(h_1), \ldots, \text{In}(h_l))$.
\end{definition}

From now on, we always assume that $\{f_1, \ldots, f_c\}$ and $\{g_1, \ldots, g_c\}$ are Gr{\"o}bner bases with respect to $\text{In}$. Then $G := k[x_1, \ldots, x_s]/(\text{In}(f_1), \ldots, \text{In}(f_c))$ and it follows that
\begin{align*}
    P_{G}^{n} = \frac{k[x_1, \ldots, x_s, \xi_1, \ldots, \xi_s]}{(\text{In}(f_1), \ldots, \text{In}(f_c), \text{In}(g_1), \ldots, \text{In}(g_c)) + (x_1 - \xi_1, \ldots, x_s - \xi_s)^{n+1}}.
\end{align*}
Let $J_1 := (\text{In}(f_1), \ldots, \text{In}(f_c), \text{In}(g_1), \ldots, \text{In}(g_c)) + (x_1 - \xi_1, \ldots, x_s - \xi_s)^{n+1}$ and $J_2 := \text{In}((f_1, \ldots, f_c, g_1, \ldots, g_c) + (x_1 - \xi_1, \ldots, x_s - \xi_s)^{n+1})$. Clearly, $J_1 \subseteq J_2$. We then have the following surjective $G$-module homomorphism
\begin{align*}
    \psi: P_{G}^{n} \longrightarrow \text{gr}(P_{R}^{n})
\end{align*}
Let $m \in J_2 \mysetminus J_1$. Then we write 
\begin{align*}
    m = \text{In}\bigg(\sum_{i}a_if_i + \sum_{j}b_jg_j + \sum_{|\alpha|=n+1 }r_{\alpha}(x_1 - \xi_1)^{\alpha_1}\cdots(x_s - \xi_s)^{\alpha_s}\bigg).
\end{align*}
for some $a_i, b_j, r_{\alpha} \in k[x_1, \ldots, x_s, \xi_1, \ldots, \xi_s]$, and assume that
\begin{align*}
    \sum_{i}\text{In}(a_i)\text{In}(f_i) + \sum_{j}\text{In}(b_j)\text{In}(g_j) + \sum_{|\alpha|=n+1 }\text{In}(r_{\alpha})(x_1 - \xi_1)^{\alpha_1}\cdots(x_s - \xi_s)^{\alpha_s} = 0.
\end{align*}
Otherwise $m \in J_1$. By the construction of the initial map~(\ref{In}), $m$ is a homogeneous polynomial and $\text{deg}(m) \geq n + 1$. Then the image of $m$ in $\text{gr}(P_{R}^{n}) \otimes_G G/m_G$ lies in $(\xi_1, \ldots, \xi_s)^{n+1}$. It is easy to see that
\begin{align*}
    P_{G}^{n} \otimes_G G/m_G = \frac{k[\xi_1, \ldots, \xi_s]}{(\text{In}(g_1), \ldots, \text{In}(g_c)) + (\xi_1, \ldots, \xi_s)^{n+1}},
\end{align*}
and 
\begin{align*}
    \operatorname{gr}(P_{R}^{n}) \otimes_G G/m_G = \frac{k[\xi_1, \ldots, \xi_s]}{\text{In}((g_1, \ldots, g_c) + (\xi_1, \ldots, \xi_s)^{n+1})}.
\end{align*}
Since $\{g_1, \ldots, g_c\}$ is a Gr{\"o}bner basis and with the similar argument as above, notice that $(\text{In}(g_1), \ldots, \text{In}(g_c)) + (\xi_1, \ldots, \xi_s)^{n+1}) = \text{In}(g_1, \ldots, g_c + (\xi_1, \ldots, \xi_s)^{n+1})$. Therefore,
\begin{align}\label{eq4}
    P_{G}^{n} \otimes_G G/m_G \cong \operatorname{gr}(P_{R}^{n}) \otimes_G G/m_G
\end{align}

\section{\texorpdfstring{$\text{HN}_n$}{n} property for \texorpdfstring{$G$}{G}}
In this section, we prove the~\nameref{HNn} property for $G$. We know $G$ is a graded ring, but $G$ may not be an integral domain or even may not be reduced. So, we cannot apply Theorem~\ref{HNNgraded} to $G$. We now follow a different approach to prove~\nameref{HNn} property for $G$.
\begin{remark}\label{rem_tor}
    Let $G_{\text{red}}: = G/{\text{rad}(0)}$ be the reduction of the associated graded ring of $R$ modulo the radical ideal of $(0)$. Then we have the induced map of $G$-module homomorphism
    \begin{align*}
        P_{G}^{n} \longrightarrow P_{G_{\text{red}}}^{n}.
    \end{align*}
    It is also easy to see that if $g \in G$ is a non-zero divisor, then $[g] \in G_{\text{red}}$ is also a non-zero divisor. So, we also get a induced $G$-module homomorphism
    \begin{align*}
        dq: \widetilde{P}_{G}^{n} \longrightarrow \widetilde{P}_{G_{\text{red}}}^{n}.
    \end{align*}
\end{remark}

\begin{lemma}
$G_{\text{red}}$ satisfies the \nameref{HNn} property.
\end{lemma}

\begin{proof}
    Let $B := G_{\text{red}}$ and $m_B := ({m_G})_{\text{red}} = (m_G + \text{rad}(0))/\text{rad}(0) = m_G/\text{rad}(0)$. We can write $B = k[x_1, \ldots, x_s]/J$ where $J$ is a homogeneous ideal with $J = \text{rad}(J)$. It follows that $J = p_1 \cap \cdots \cap p_l$ where $p_1, \ldots, p_l$ are homogeneous prime ideals containing $J$. Let $B_i := B/p_i$ and $Q_i = \text{Frac}(B_i)$ be the fraction field of $B_i$. Then the total ring of fractions of $B$, denoted $Q$ is given by $Q: = Q_1 \times \cdots \times Q_l$. We have a map $B \longrightarrow Q$ which induces the map $g: P_{B}^n \longrightarrow P_{Q}^n = \bigoplus_{i=1}^{l}P_{Q_i}^n$, and $\text{im}(g) = \widetilde{P}_{B}^{n}$ 
    We know that there exists $i_0 \in \{1, \ldots, l\}$ such that $\dim B_{i_0} = \dim B := d$. Since $B_{i_0} = B/p_{i_0}$, we have a $B$-module homomorphism $P_{B}^{n} \longrightarrow P_{B_{i_o}}^{n}$. We also have a map $B_{i_0} \longrightarrow Q_{i_0}$ which induces the map $g_{i_0}: P_{B_{i_0}}^n \longrightarrow P_{Q_{i_0}}^n$, and $\text{im}(g_{i_0}) = \widetilde{P}_{B_{i_0}}^{n}$. Notice that we have a commutative diagram
    \[ 
\begin{tikzcd}
P_{B}^{n} \arrow{r}{g} \arrow[swap]{d} & P_{Q}^{n} \arrow{d} \\%
P_{B_{i_o}}^{n} \arrow{r}{g_{i_0}}& P_{Q_{i_o}}^{n}.
\end{tikzcd}
\]
It follows that we have a $B$-module homomorphism ${dq}_0: \widetilde{P}_{B}^{n} \longrightarrow \widetilde{P}_{B_{i_0}}^{n}$. By the construction in the Section~\ref{construction}, we have the following commutative diagram:
\[ 
\begin{tikzcd}
P_{A}^n \otimes_A B/m_B \arrow{r}{\overline{d\varphi}} \arrow[swap]{d} & \widetilde{P}_{B}^{n} \otimes_B B/m_B \arrow{d} \\%
P_{A}^n \otimes_A B_{i_0}/m_{B_{i_0}} \arrow{r}{\overline{d\varphi}_0}& \widetilde{P}_{B_{i_0}}^{n} \otimes_{B_{i_0}} B_{i_0}/m_{B_{i_0}},
\end{tikzcd}
\]
    where $A = k[x_1, \ldots, x_d]$ is the generic linear projection as defined in the arc lemma~\ref{arc}. The left vertical map in the above diagram is an isomorphism, and by Remark~\ref{graded}, $\overline{d\varphi}_0$ is injective. It follows that $\overline{d\varphi}$ is injective.
\end{proof}

\begin{prop}\label{graded*}
$G$ satisfies the \nameref{HNn} property.
\end{prop}

\begin{proof}
    The Remark~\ref{rem_tor} and the construction in the Section~\ref{construction} gives us the following commutative diagram
    \[ 
\begin{tikzcd}
{P_{A}^n} \otimes_A G/m_G \arrow{r}{\overline{d\phi}} \arrow[swap]{d} & {\widetilde{P}_{G}^{n}} \otimes_G G/m_G \arrow{d} \\%
{P_{A}^n} \otimes_A G_{\text{red}}/({m_G})_{\text{red}} \arrow{r}{\overline{d\varphi}}& {\widetilde{P}_{G_{\text{red}}}^{n}} \otimes_{G_{\text{red}}} G_{\text{red}}/({m_G})_{\text{red}}.
\end{tikzcd}
\]
Here, the left vertical map is an isomorphism, and by the above lemma, we know $\overline{d\varphi}$ is injective. It follows that $\overline{d\phi}$ is injective.
\end{proof}

\begin{remark}
    We know a classical result for the module of Kähler differentials saying that if $\Omega_R$ is free, then $R$ is regular. We give a analogous result for the module of principal parts. This was previously obtained by Benner-Jeffries-Nuñez \cite[see Theorem 10.2]{brenner2019quantifying}.
 
\end{remark}

\begin{cor}
    Let $R$ be as in Setup~\ref{setup} of Krull dimension $d$. Suppose $P_{R}^{n}$ is free $R$-module of rank $D := \binom{n+d}{d}$, then $R$ is regular.
\end{cor}

\begin{proof}
    Consider the following commutative diagram
    \[ 
\begin{tikzcd}
{P_{A}^n} \otimes_A G/m_G \arrow{r}{dp} \arrow[swap]{d} & {P_{G}^{n}} \otimes_G G/m_G \arrow{d} \\%
{P_{A}^n} \otimes_A G/{m_G} \arrow{r}{\overline{dp}}& \widetilde{P}_{G}^n \otimes_G G/m_G.
\end{tikzcd}
\]
Then from Proposition~\ref{graded*}, $dp$ is injective. Recall from~(\ref{eq}) that
\begin{align*}
    P_{R}^{n} = \frac{k[x_1, \ldots, x_s, \xi_1, \ldots, \xi_s]_{(x_1, \ldots, x_s, \xi_1, \ldots, \xi_s)}}{(f_1, \ldots, f_c, g_1, \ldots, g_c) + (x_1 - \xi_1, \ldots, x_s - \xi_s)^{n+1}}.
\end{align*}
And, we get
\begin{align*}
    P_{R}^{n} \otimes_R R/m_R = \frac{k[\xi_1, \ldots, \xi_s]_(\xi_1, \ldots, \xi_s)}{(g_1, \ldots, g_c) + (\xi_1, \ldots, \xi_s)^{n+1}}.
\end{align*}
It is then easy to see that the associated graded ring of $P_{R}^{n} \otimes_R R/m_R$ with respect to $(\xi_1, \ldots, \xi_s)$-adic filtration can be written as
\begin{align*}
    \operatorname{gr}(P_{R}^{n} \otimes_R R/m_R) = \frac{k[\xi_1, \ldots, \xi_s]}{\text{In}((g_1, \ldots, g_c) + (\xi_1, \ldots, \xi_s)^{n+1})}.
\end{align*}
Here, $\text{In}$ is the same as in~(\ref{In}).
It follows from~(\ref{eq4}) that
\begin{align*}
   \operatorname{gr}(P_{R}^{n} \otimes_R R/m_R) \cong \operatorname{gr}(P_{R}^{n}) \otimes_G G/m_G\cong P_{G}^{n} \otimes_G G/m_G
\end{align*}
We then get the following commutative diagram, and notice that the vertical maps are not linear
\[ 
\begin{tikzcd}
P_{A}^n \otimes_A R/m_R \arrow{r}{d{\psi}} \arrow[swap]{d}{\text{In}} & P_{R}^{n} \otimes_R R/m_R \arrow{d}{\text{In}} \\%
P_{\text{gr}(A)}^n \otimes_{\text{gr}(A)} G/m_G \arrow{r}{d\phi}& P_{G}^{n} \otimes_G G/m_G.
\end{tikzcd}
\]
Suppose $u \in P_{A}^n \otimes R/m_R$, then ${\text{In}}(u) \neq 0$. We know $d\phi$ is injective,and thus, $d{\psi}$ is injective. But $\dim(P_{A,+}^n \otimes_A R/m_R) = \dim(P_{R,+}^{n} \otimes_R R/m_R) = D-1$, and so $d{\psi}$ is an isomorphism. This is equivalent to saying that
\begin{align*}
    P_{A,+}^n \otimes_A R/m_R \longrightarrow P_{R,+}^{n} \otimes_R R/m_R
\end{align*}
is an isomorphism. Further, it follows from Nakayama's lemma that the $R$-module $P_{R,+}^{n}$ is freely generated by $\big\{{{d^n}({x_1})}^{\alpha_1} \cdots {{d^n}({x_d})}^{\alpha_d} \mid 1 \leq |\alpha| \leq n\big\}$. Furthermore, we have a surjective $R$-module homomorphism
\begin{align*}
    P_{R,+}^{n} = I_R/I_R^{n+1} \longrightarrow \Omega_{R} = I_R/I_R^2.
\end{align*}
It follows that the $R$-module $\Omega_{R}$ is generated by $\big\{d({x_1}), \cdots, d({x_d})\big\}$. But $\operatorname{rank}(\Omega_{R}) = d$. Thus, $\Omega_{R}$ is a free $R$-module, and so $R$ is regular.
\end{proof}



\section{Higher Nobile's theorem}

We are now ready to prove the higher Nobile's theorem for the graded case, and also for any algebraic variety when $n = 2$.
\begin{thm}\label{Higher_Nash_graded}(Graded case)
   Let $X$ be a homogeneous algebraic variety, and $R$ be as in Setup~\ref{setup}. Suppose $\Nashn(X) \cong X$ for some $n \in \mathbb{N}$, then
   \begin{enumerate}[label=(\arabic*)]
        \item if $\operatorname{char}(k) = 0$, then $X$ is a non-singular variety.
        \item if $X$ is normal and $\operatorname{char}(k) > 0$, then $X$ is a non-singular variety.
    \end{enumerate}
   \end{thm}

\begin{proof}
     By the Theorem~\ref{NashntoNashO}, \cite[Theorem 2]{MR0409462} and \cite[Theorem 3.10] {MR4382471}, it concludes the proof of (1) and (2).
\end{proof}

\begin{thm}($n = 2$)\label{Higher_Nash_n=2}
    Let $X$ be an algebraic variety over an algebraically closed field $k$ with $\operatorname{char}(k) = 0$. Suppose $\NashT(X) \cong X$, then $X$ is a non-singular variety.
\end{thm}

\begin{proof}
    By the Theorem~\ref{NashntoNashO} and \cite[Theorem 2]{MR0409462}, $X$ is a non-singular variety.
\end{proof}

\subsection*{Acknowledgements}
I would like to sincerely thank my advisor, Dr. Roi Docampo, for his constant support and valuable guidance throughout this research. His knowledge and advice were essential in helping me develop and complete this paper. I also want to thank Daniel Duarte for his feedback and valuable comments about the previous versions of this paper.






\bibliographystyle{alpha-links}
\bibliography{references}

\end{document}

%% file: article-style.tex

\pagestyle{fancy}
\fancyhead{}
\fancyfoot{}
\fancyhead[LE,RO]{\small \thepage}
\fancyhead[RE]{\small \nouppercase{\rightmark}}
\fancyhead[LO]{\small \nouppercase{\leftmark}}

\setcounter{tocdepth}{1}

\makeatletter


\global\BR@BackrefAlttrue

\renewcommand*{\backrefalt}[4]{%
    \ifcase #1 {\color{red}Not cited}%
          \or Cited on page~#2%
          \else Cited on pages~#2%
    \fi%
    .
}


%

\def\maketitle{\par
  \@topnum\z@ 
  \@setcopyright
  \thispagestyle{empty}
  \ifx\@empty\shortauthors \let\shortauthors\shorttitle
  \else \andify\shortauthors
  \fi
  \@maketitle@hook
  \begingroup
  \@maketitle
  \toks@\@xp{\shortauthors}\@temptokena\@xp{\shorttitle}%
  \toks4{\def\\{ \ignorespaces}}
  \edef\@tempa{%
    \@nx\markboth{\the\toks4
      \@nx\MakeUppercase{\the\toks@}}{\the\@temptokena}}%
  \@tempa
  \endgroup
  \c@footnote\z@
    \renewcommand{\footnoterule}{%
      \kern -3pt
      \hrule width \textwidth height .5pt
      \kern 2pt
    }
  {
    \renewcommand\thefootnote{}
    \vspace{-2em}
    \footnote{
      \par\vspace{-1.2em}\noindent%
      \setlength{\parindent}{0pt}%
      \def\@footnotetext##1{\noindent{\footnotesize##1}\par}%
      \let\@makefnmark\relax  \let\@thefnmark\relax
      \ifx\@empty\@date\else \@footnotetext{\@setdate}\fi
      \ifx\@empty\@subjclass\else \@footnotetext{\@setsubjclass}\fi
      \ifx\@empty\@keywords\else \@footnotetext{\@setkeywords}\fi
      \ifx\@empty\thankses\else \@footnotetext{%
        \@setthanks}%
      \fi
    }
    \addtocounter{footnote}{-1}
  }
  \@cleartopmattertags
}

%

\def\@adminfootnotes{\@empty}

%

\def\@settitle{\begin{center}%
  \baselineskip14\p@\relax
    \bfseries
\Large
  \@title
  \end{center}%
}

%

\def\@setauthors{%
  \begingroup
  \def\thanks{\protect\thanks@warning}%
  \trivlist
  \centering\footnotesize \@topsep30\p@\relax
  \advance\@topsep by -\baselineskip
  \item\relax
  \author@andify\authors
  \def\\{\protect\linebreak}%
  \large{\authors}%
  \ifx\@empty\contribs
  \else
    ,\penalty-3 \space \@setcontribs
    \@closetoccontribs
  \fi
  \endtrivlist
  \endgroup
}

%

\def\@setaddresses{\par
  \nobreak \begingroup
\footnotesize
  \def\author##1{\end{minipage}\hskip 2em \begin{minipage}[t]{.5\textwidth minus
  1em}\raggedright%
    ~\\[2em]{\bf##1}\\[.5em]%
  }%
  \interlinepenalty\@M
  \def\address##1##2{\begingroup
    {\ignorespaces##2}\endgroup\\[.5em]}%
  \def\curraddr##1##2{\begingroup
    \@ifnotempty{##2}{\nobreak\indent\curraddrname
      \@ifnotempty{##1}{, \ignorespaces##1\unskip}\/:\space
      ##2\par}\endgroup}%
  \def\email##1##2{\begingroup
    \@ifnotempty{##2}{\nobreak\indent
      \@ifnotempty{##1}{, \ignorespaces##1\unskip}
      \ttfamily##2\par}\endgroup}%
  \def\urladdr##1##2{\begingroup
    \def~{\char`\~}%
    \@ifnotempty{##2}{\nobreak\indent\urladdrname
      \@ifnotempty{##1}{, \ignorespaces##1\unskip}\/:\space
      \ttfamily##2\par}\endgroup}%
  \setlength{\parindent}{0pt}%
  \vfill%
  {
  \hskip -2em%
  \begin{minipage}{0mm}
  \addresses
  \end{minipage}
  }
  \endgroup
}

%

\renewcommand{\author}[2][]{%
  \ifx\@empty\authors
    \gdef\authors{#2}%
    \g@addto@macro\addresses{\author{#2}}%
  \else
    \g@addto@macro\authors{\and#2}%
    \g@addto@macro\addresses{\author{#2}}%
  \fi
  \@ifnotempty{#1}{%
    \ifx\@empty\shortauthors
      \gdef\shortauthors{#1}%
    \else
      \g@addto@macro\shortauthors{\and#1}%
    \fi
  }%
}
\edef\author{\@nx\@dblarg
  \@xp\@nx\csname\string\author\endcsname}

%

\def\@secnumfont{\@empty}

%

\def\section{\@startsection{section}{1}%
  \z@{.7\linespacing\@plus\linespacing}{.5\linespacing}%
  {\large\bfseries\centering}}

\makeatother